\documentclass[10pt,notitlepage]{article}
\usepackage{amsmath, amsthm, amssymb}
\usepackage{graphicx}%authblk
\usepackage[v2]{xy}
\xyoption{all}
\numberwithin{equation}{section}

\theoremstyle{plain}
\newtheorem{theorem}[equation]{Theorem}
\newtheorem*{thm}{Theorem}
\newtheorem{proposition}[equation]{Proposition}
\newtheorem{lemma}[equation]{Lemma}
\newtheorem{corollary}[equation]{Corollary}

\theoremstyle{definition}

\newtheorem{remark}[equation]{Remark}

\newtheorem*{ques}{Question}

\newcommand\wt{\widetilde}
\newcommand\ov{\overline}

\newcommand\surj{\twoheadrightarrow}
\long\def\symbolfootnote[#1]#2{\begingroup%
\def\thefootnote{\fnsymbol{footnote}}\footnote[#1]{#2}\endgroup}

\begin{document}
\title{{Totalization of simplicial homotopy types}}
\date{}
\author{C. Ogle, A. Salch}
\maketitle
\begin{abstract} We identify the obstructions for the functoriality and the uniqueness of the totalization functor, (partially) defined on the category of simplicial objects in the homotopy category of a stable model category, and we use a result from the cyclic homology of group algebras to show they can be non-zero.
\end{abstract}

\symbolfootnote[0]{2000 {\it Mathematics Subject Classification}. Primary 18G55; Secondary 18G30, 18G35, 18G40, 55U10, 55U15, 55U35, 55U40.}
\symbolfootnote[0]{{\it Key words and phrases}. Homotopy category, closed model category, simplicial object.}
\vskip.5in

\section{Introduction/Statement of results}
\vskip.4in

Let $\cal C$ denote a stable model category (our motivating example is the category of bounded-below chain complexes over $\mathbb C$). Then the associated homotopy category $Ho({\cal C})$ is triangulated in a natural way. Let $S_{\bullet}{\cal C}$ denote the category of simplicial objects over $\cal C$; the objects in this category are functors $F:\Delta^{op}\to {\cal C}$, with morphisms given by natural transformations of such functors. Given such a simplicial object $C_{\bullet} = \{[n]\mapsto C_{n}\}_{n\ge 0}$, one can {\em totalize} it. (For our decision to refer to this construction as ``totalization'' and not its more common name, ``geometric
realization,'' see Remark~\ref{remark on terminology}.) We then have a totalization functor 
$Tot:S_{\bullet}{\cal C}\to {\cal C}$. % which sends a simplicial object $C_{*,\bullet}$ to the total complex of $C_{*,\bullet}$: $C_{*,\bullet}\mapsto Tot_*(C_{*,\bullet}) := Tot_*(C_{**})$, where $C_{**}$ denotes the associated bicomplex of $C_{*,\bullet}$. 
Using this functor, one has the following fundamental definition:
\begin{itemize}
%\item $H_*(C_{*,\bullet}) := H_*(Tot_*(C_{*,\bullet}))$;
\item A simplicial morphism $\phi_{\bullet}:C_{\bullet}\to D_{\bullet}$ is a weak equivalence in $S_{\bullet}{\cal C}$ iff $Tot(\phi_{\bullet}):Tot(C_{\bullet})\to Tot(D_{\bullet})$ is a weak equivalence in $\cal C$.
\end{itemize}

This definition of weak equivalence yields a closed model structure on $S_{\bullet}{\cal C}$ compatible with that on $\cal C$ via the functor $Tot$\footnote{This is in distinction to [CL], where Reedy shows that the category of simplicial objects over a model category admits a model structure when using the more restrictive notion of degreewise weak equivalence}. Now one can also consider the simplicial category $S_{\bullet}Ho({\cal C})$ of simplicial objects over the homotopy category, and two natural questions to ask regarding totalization are:
\vskip.2in

{\bf\underbar{Question 1}} (Existence) 
\vskip.1in
a) (weak form) Let $\ov{C}_{\bullet}$ be an object in $S_{\bullet}Ho({\cal C})$. Does $Tot(\ov{C}_{\bullet})$ exist? 
 (If the answer is ``yes,'' we will say that $\ov{C}_{\bullet}$ is {\em totalizable.})
\vskip.05in
b) (strong form)  Let $\ov{C}_{\bullet}$ be an object in $S_{\bullet}Ho({\cal C})$. Does there exist an object $C_{\bullet}$ of $S_{\bullet}{\cal C}$ with $\ov{C}_{\bullet} = [C_{\bullet}]$?
\vskip.2in

This issue has been partially addressed by B\"okstedt and Neeman in [BN]. Precisely, in [\S 3, BN], the authors show that the natural construction of $Tot(\ov{C}_{\bullet})$ as a homotopy colimit in the (triangulated) homotopy category $Ho({\cal C})$ can be realized iff a sequence of first-order Toda brackets vanishes (cf. [K]). In fact their argument shows a bit more, so we recall their setup. Let $\ov{C}_*$ denote the chain complex object in $Ho({\cal C})$ associated to $\ov{C}_{\bullet}$ by taking alternating sums of face sums (this constructions makes sense since triangulated categories are additive, so we can add and subtract the face maps from each other). For each $n\ge 0$, let $T_n$ denote the (hypothetical) total object of the \lq\lq n-skeleton\rq\rq\ $\ov{C}_n\to \ov{C}_{n-1}\to\dots \to \ov{C}_0$. Obviously $T_0 = \ov{C}_0$, and $T_1$ is the mapping cone of $\ov{C}_1\to \ov{C}_0$, described by the triangle
\[
\ov{C}_1\to \ov{C}_0\to T_1\to \Sigma \ov{C}_1.
\]
The two nulhomotopies of the composition $\ov{C}_2\to\ov{C}_1\to \ov{C}_0\to T_1$ allow one to define a map $\Sigma \ov{C}_2\to T_1$. We let $T_2$ be the mapping cone of this map. One then considers the composite $\phi_3:\Sigma \ov{C}_3\to \Sigma \ov{C}_2\to T_1$, which is the first Toda bracket associated to the homotopy chain complex $\ov{C}_*$. If $\phi_3=0$, then we get a map $\Sigma^2 \ov{C}_3 \to T_2$, and we let $T_3$ be its mapping cone. Then one considers the composite
$\phi_4:\Sigma^2\ov{C}_4\to \Sigma^2\ov{C}_3\to T_2$. If $\phi_4=0$, then we get a map $\Sigma^3 \ov{C}_4 \rightarrow T_3$, and we let $T_4$ be its mapping cone.
At this point the inductive construction should be clear. The argument of [BN] leads to the following, which essentially answers part a) of Question 1.

\begin{thm} \label{bokstedt-neeman thm} [{[BN]}] Let $\ov{C}_*$, $T_i, i = 0,1$ be as above. Suppose that for $k<n$, $T_k$ exists and that there is a triangle
\[
T_{k-1}\to T_k\to \Sigma^k \ov{C}_k\to \Sigma T_{k-1}
\]
Then the $\text{(n-2)}^{\text{nd}}$ Toda bracket associated to the complex $\ov{C}_*$
\[
\phi_n:\Sigma^{n-2} \ov{C}_n\to T_{n-1}
\]
is defined, and $T_n$ can be constructed from this data iff $\phi_n=0$. The totalization $T_{\infty}$ exists iff $T_n$ can be constructed in this fashion for each $n\ge 0$, in which case one has $T_{\infty} := \underset{\xrightarrow{\hspace*{1cm}}}{hocolim}\ T_n$.  
\end{thm}

Thus the total complex $T_{\infty}$ of $\ov{C}_*$, if it exists, is naturally equipped with a skeletal filtration $\{T_n\}$, for which the associated graded object satisfies $Gr_n(T_{\infty})/Gr_{n-1}(T_{\infty}) = \Sigma^n \ov{C}_n$ for each $n$.  Note that this construction does not resolve the existence issue raised by part b), for which the obstructions seem to be a bit more subtle. Nevertheless, following this train of thought and motivated by the question posed by B\"okstedt and Neeman on [p. 219, BN], we have
\vskip.2in

{\bf\underbar{Question 2a}} (Uniqueness) Let $C_{\bullet}$, $D_{\bullet}$ be two objects of $S_{\bullet} \cal C$, with the corresponding objects $\{[n]\mapsto [C_{n}]\}_{n\ge 0}$ resp\@. $\{[n]\mapsto [D_{n}]\}_{n\ge 0}$ in $S_{\bullet}Ho({\cal C})$ denoted by $[C_{\bullet}]$ resp\@. $[D_{\bullet}]$. If $\ov{f}_{\bullet}:[C_{\bullet}]\to [D_{\bullet}]$ is a simplicial map which is the identity in each degree, does $\ov{f}_{\bullet}$ extend to a filtration-preserving equivalence $g:Tot(C_{\bullet})\to Tot(D_{\bullet})$?
\vskip.2in

More generally, one can ask when a map of totalizable simplicial homotopy types induces a map of totalizations:

{\bf\underbar{Question 2b}} (Functoriality) Let $C_{\bullet}$, $D_{\bullet}$ be two objects of $S_{\bullet} Ho( \cal C)$ whose totalizations both exist. If $f_{\bullet}:C_{\bullet}\to D_{\bullet}$ is a simplicial map, does $f_{\bullet}$ extend to a filtration-preserving morphism $g_*:Tot(C_{\bullet})\to Tot(D_{\bullet})$?
\vskip.2in

As observed in [\S 3, BN], it is relatively easy to construct simplicial objects in $Ho({\cal C})$ for which the first possibly non-trivial Toda bracket is in fact non-trivial, from which one concludes that the totalization of $\ov{C}_*$ in general does not exist. There is, then, the secondary issue of the uniqueness of the totalization, framed as in Question 2a, and of functoriality, as in Question 2b. Our first main result is the identification of a series of  obstructions associated to a morphism of simplicial objects in $Ho({\cal C})$ whose vanishing provides necessary and sufficient conditions for the existence of a filtration-preserving morphism of totalizations to exist, as specified by these two questions. These obstructions may be thought of as the first and higher-order Toda brackets associated to a map of totalizable simplicial objects in the homotopy category $Ho (\cal C)$ (as defined below). In section 3, we show that a well-known equivalence used to compute the cyclic homology of group algebras, simplicialized, provides an explicit example of the non-triviality of these Toda brackets, and thus an explicit case of when the totalization, even when it exists, is not unique. Precisely, we show, for $\mathcal{C}$ the category of bounded-below chain complexes over $\mathbb{C}$:

\begin{theorem} There are simplicial chain complexes $C_{*,\bullet}$, $D_{*,\bullet}$ in $S_*{\cal C}$ such that
\begin{itemize}
\item[P1] There exists a morphism of graded chain complexes $\{\phi_n:C_{*,n}\to D_{*,n}\}_{n\ge 0}$ with $\phi_n$ a weak equivalence (i.e., quasi-isomorphism) for each $n$,
\item[P2] for each morphism $\alpha\in Hom_{\Delta}([m],[n])$, there is a canonical chain homotopy
\[
\phi_m\circ C_{*,\bullet}(\alpha)\simeq D_{*,\bullet}(\alpha)\circ \phi_n : C_{*,n}\to D_{*,m},
\]
\item[P3] and $H_*(Tot_* C_{*,\bullet})\ne H_*(Tot_* D_{*,\bullet})$.
\end{itemize}
\end{theorem}

We would like to thank the referee for bringing [AB] to our attention, which considers issues related to the ones discussed here.
\vskip.4in

\section{Maps of simplicial homotopy types and higher Toda brackets}
\vskip.3in

\begin{remark}\label{remark on terminology}
The terminologically fastidious reader might insist that one {\em totalizes} a {\em cosimplicial} object, and one
{\em geometrically realizes} a {\em simplicial} object. To that way of thinking, we should be writing ``geometric realization'' rather than ``totalization'' throughout
this paper, since we are considering simplicial objects and not cosimplicial objects. 
The reason we have chosen to stick with the term ``totalization'' is that the main example in the last section of this paper, due to the first author,
really involves a totalization in the most classical sense: a totalization of a double complex. A suitably modern statement of the Eilenberg-Zilber theorem
is that there is a quasi-isomorphism between, on the one hand, the classical totalization of the alternating sum double complex
of any simplicial chain complex of abelian groups; and on the other hand, the geometric realization of that same simplicial chain complex of abelian groups.
So there is good precedent for our usage of the term ``totalization'' and it is motivated by our main example.
\end{remark}

\vskip 0.1in

Now we come around to the main question of this paper: suppose $C_{\bullet},D_{\bullet}$ are simplicial objects in the homotopy category $Ho ( \mathcal{C})$ of a stable model category $\mathcal{C}$.
Suppose further that $C_{\bullet},D_{\bullet}$ are both {\em totalizable,} that is, the totalizations of $C_{\bullet}$ and $D_{\bullet}$ in the sense of
Thm.~\ref{bokstedt-neeman thm} both exist. Finally, suppose we have a map $f_{\bullet}: C_{\bullet}\rightarrow D_{\bullet}$ of simplicial objects in
$Ho (\mathcal{C})$. We sometimes call $f_{\bullet}$ a {\em map of simplicial homotopy types.} 
Does $f_{\bullet}$ induce a map of totalizations $Tot  C_{\bullet}\rightarrow Tot  D_{\bullet}$?
\vskip.1in

To $C_{\bullet}$ we can associate an alternating sum chain complex object, for which we will write $C_*$. Let ${\cal F}_*C_{*}$ denote the skeletal filtration of $C_{*}$; that is, ${\cal F}_nC_{*} = \{C_{k}\}_{0\le k\le n}$, with ${\cal T}_nC_{*} := Tot\left({\cal F}_nC_{*}\right)$. Finally, for $1\le l\le n, n\ge 0$, let $Gr^l_n{\cal F}C_{*} = \left({\cal F}_nC_{*}\right)/\left({\cal F}_{n-l}C_{*}\right)$; similarly for $D_{*}$. The object $Tot(C_{*})$ is filtered by $\{ Tot ({\cal F}_n C_*) \}_{n\ge 0}$, and for each $n$
\[
Gr^1_n{\cal F}C_{*} = \left(Tot ({\cal F}_n C_*)\right)/\left(Tot ({\cal F}_{n-1} C_*)\right) = \Sigma^n C_{n} .
\]
By ``filtration'' here we mean we have a natural sequence of maps 
\[ 0 \simeq Tot\left({\cal F}_{-1}C_{*}\right) \to Tot\left({\cal F}_0C_{*}\right) \to Tot\left({\cal F}_1C_{*}\right) \to Tot\left({\cal F}_2C_{*}\right) \to \dots \]
whose homotopy colimit is $Tot(C_*) = Tot(C_{\bullet})$. 
(We are writing $Tot$ for what is really a geometric realization, and this sequence of maps is sometimes called the
{\em geometric realization tower,} a name which is more convincing if you draw the sequence vertically. We draw it this way in 
diagram~\ref{infinite geom realization tower for ss}, below.)

Questions 2a and 2b of the previous section may be reformulated as:

\begin{ques} When does there exist a filtration-preserving homomorphism of totalizations $g_*:Tot(C_{*})\to Tot(D_{*})$ with $g_n\simeq f_n:Gr^1_n{\cal F}C_{*}\to Gr^1_n{\cal F}D_{*}$ for each $n$?
\end{ques}

As we shall see, there is a naturally defined hierarchy of obstructions associated to the existence of such a map, and  that even the first-order obstructions are in general non-zero. To describe them, first note that by Dold-Kan, the homotopy commutativity of the above diagram is equivalent to the statement that
\begin{equation}\label{diag:diag4.3.2}
\xymatrix{
\ar@{}[dr]| *+[F-:<200pt>]{\simeq}
C_{n}\ar[d]^{d^{C}_{n}}\ar[r]^{f_n} &
D_{n}\ar[d]^{d^{D}_{n}}\\
C_{n-1}\ar[r]^{f_{n-1}} &
D_{n-1}
}
\end{equation}
commutes in the homotopy category $Ho (\mathcal{C})$ for each $n\ge 1$, with the vertical differentials given as the alternating sum of the face maps from dim.\ $n$ to dim.\ $n-1$. The first step to constructing $g_{\bullet}$ (or, equivalently, $g_*$) is provided by

\begin{proposition} \label{map on gr^2} For each $n\ge 1$ there exists a filtration-preserving map
\[
Tot\left(Gr^2_n{\cal F}C_{*}\right)\overset{h}\longrightarrow Tot\left(Gr^2_n{\cal F}D_{*}\right)
\]
which on the subquotients $Gr^1_m{\cal F}C_{*}$ agree with $f_m$  ($m = n,n-1$).
If $f_m$ is a weak equivalence for every $m$, then so is $h$.
\end{proposition}
\begin{proof} 
Since $Tot \left(Gr^2_n{\cal F}C_{*}\right)$ is just the cofiber of the map $d^{C}_{n+2}: \Sigma^{n+1}C_{n+2}\rightarrow \Sigma^n C_{n+1}$, the map $h$
is just the map induced in cofibers:
\[ \xymatrix{ \Sigma^{n+1}C_{n+2}\ar[rr]^{\Sigma^{n+1}d^{C}_{n+2}} \ar[d]_{\Sigma^{n+1}f_{n+2}} & & \Sigma^{n+1}C_{n+1}\ar[d]^{\Sigma^{n+1}f_{n+1}}\ar[r] & Tot \left(Gr^2_n{\cal F}C_{*}\right)\ar@{-->}[d] \\
\Sigma^{n+1}D_{n+2}\ar[rr]^{\Sigma^{n+1}d^{D}_{n+2}} &  & \Sigma^{n+1}D_{n+1}\ar[r] & Tot \left(Gr^2_n{\cal F}D_{*}\right) , }\]
where the left-hand square commutes in $Ho( \mathcal{C})$. That this map of cofibers exists is part of what one proves in the usual process of showing that
the homotopy category of a stable model category is triangulated (or, more generally, that the homotopy category of a pointed model category is pretriangulated);
one can consult Prop.~6.3.5 of [H].
\vskip.1in

We include an appealingly explicit construction of this map in the special case that $\mathcal{C}$ is the category of bounded-below 
chain complexes of $R$-modules, for some ring $R$. 
By assumption, for each $n$ there exists a chain homotopy $s(1)_{*n}:C_{*n}\to D[1,-1]_{*n} := D_{(*+1)(n-1)}$ with
\begin{equation}\label{eqn:stage1hom}
f_{*n-1}\circ d^{C}_{*n} - d^{D}_{*n}\circ f_{*n} = d^{1D}_{(*+1)(n-1)}\circ s(1)_{*n} + s(1)_{(*-1)n}\circ d^{1C}_{*n}
\end{equation}
where $d^{1C}_{*m}:C_{*m}\to C_{(*-1)m}$ denotes the differential in the first coordinate (similarly for $D_{**}$). Now $Tot_{k+n}\left({\cal F}_nC_{**}/{\cal F}_{n-2}C_{**}\right)\cong C_{(k+1)(n-1)} \oplus C_{kn}$ and similarly for $D_{**}$. By
equation (\ref{eqn:stage1hom}) above, the map
\begin{gather}
C_{(k+1)(n-1)}\oplus C_{kn}\to D_{(k+1)(n-1)}\oplus D_{kn},\\
\left(x_1,x_2\right)\mapsto \left(f_{(k+1)(n-1)}(x_1) + s(1)_{kn}(x_2),f_{kn}(x_2)\right)
\end{gather}
defines a chain map of total complexes $Tot_*\left(Gr^2_n{\cal F}C_{**}\right)\to Tot_*\left(Gr^2_n{\cal F}D_{**}\right)$ which, by construction, agrees with $f_{\bullet}$ on the subquotients $Gr^1_m{\cal F}C_{**}$ for $m = n,n-1$.
\end{proof}
\vskip.1in

Now since the square
\[ \xymatrix{   C_{n+1} \ar[d]_{ d^{C}_{n}} \ar[r]^{ f_{n+1}} & 
                  D_{n+1} \ar[d]^{ d^{D}_{n}} \\ 
                C_{n}  \ar[r]^{ f_{n}} & 
                  D_{n}  }\]
commutes {\em up to homotopy,} the difference $ f_n \circ  d^{C}_n -  d^{D}_n\circ  f_{n+1}$
is a nulhomotopic map $  C_{n+1}\rightarrow   D_n$. Of course, the same is true with $n$ replaced by $n+1$ throughout.
As a consequence, in the diagram
\[ \xymatrix{ 
                C_{n+2} \ar[d]_{ d^{C}_{n+1}} \ar[r]^{ f_{n+2}} & 
                  D_{n+2} \ar[d]^{ d^{D}_{n+1}} \\ 
                C_{n+1} \ar[d]_{ d^{C}_{n}} \ar[r]^{ f_{n+1}} & 
                  D_{n+1} \ar[d]^{ d^{D}_{n}} \\ 
                C_{n}  \ar[r]^{ f_{n}} & 
                  D_{n},  }\]
we have {\em two} nulhomotopies of the difference map
\[ f_n \circ  d^{C}_n\circ d^{C}_{n+1} -  d^{D}_n\circ d^{D}_{n+1}\circ  f_{n+2}: C_{n+2}\rightarrow D_n,\]
which specifies a map $T(2,n;f_{\bullet}): \Sigma C_{n+2}\rightarrow D_n$. The homotopy class $[T(2,n;f_{\bullet})]\in [\Sigma C_{n+2}, D_n]$ is the 
obstruction to extending Prop.~\ref{map on gr^2} to $Gr^3_n$:
\begin{proposition}\label{gr^3 map}
For each fixed $n\ge 2$ there exists a filtration-preserving weak equivalence
\[
Tot\left(Gr^3_n{\cal F}C_{*}\right)\overset{\simeq}\longrightarrow Tot\left(Gr^3_n{\cal F}D_{*}\right)
\]
agreeing with $f_m$ on the subquotients $Gr^1_m{\cal F}C_{*}$ ($m = n,n-1,n-2$) iff $[T(2,n;f_{\bullet})] = 0$.
\end{proposition}
\begin{proof}
We note that $Tot\left( Gr^3_n{\cal F}C_*\right)$ sits in a tower of cofiber sequences
\[\xymatrix{
 \Sigma^{n+1}C_{n+2} \ar[r] & \Sigma^{n+1} C_{n+1} \ar[d] \\
 \Sigma^{n+2}C_{n+3} \ar[r] & Tot ( Gr^2_{n}{\cal F}C_*) \ar[d] \\
                            & Tot ( Gr^3_{n}{\cal F}C_*) .}\]
From Prop.~\ref{map on gr^2}, we know that we have a map defined up to homotopy on the top portion of this tower and its analogue for $D_{\bullet}$:
\[\xymatrix{
 \Sigma^{n+1}C_{n+2} \ar[r]\ar@/^1pc/[rr] & \Sigma^{n+1} C_{n+1} \ar@/_1pc/[rr] \ar[d] & \Sigma^{n+1}D_{n+2} \ar[r] & \Sigma^{n+1} D_{n+1}  \ar[d] \\
 \Sigma^{n+2}C_{n+3} \ar[r] & Tot ( Gr^2_{n}{\cal F}C_*) \ar[d]\ar@/_1pc/[rr] &  \Sigma^{n+2}D_{n+3} \ar[r] & Tot ( Gr^2_{n}{\cal F}D_*) \ar[d] \\
                            & Tot ( Gr^3_{n}{\cal F}C_*) &  & Tot ( Gr^3_{n}{\cal F}D_*) .}\]
We want to extend this map, up to homotopy, to the bottoms of the towers. This is equivalent to asking that the square
\begin{equation}\label{diagram 1} \xymatrix{
 \Sigma^{n+2}C_{n+3} \ar[r]^{\Sigma^{n+2}f_{n+3}} \ar[d] & \Sigma^{n+2} D_{n+3} \ar[d] \\
 Tot ( Gr^2_n{\cal F}C_* ) \ar[r] & Tot (Gr^2_n{\cal F}D_*) }\end{equation}
commute up to homotopy. Recall that the map $\Sigma^{n+2}C_{n+3}\rightarrow Tot ( Gr^2_n{\cal F}C_*)$ arises
from the two nulhomotopies of the map $\Sigma^{n+1}C_{n+3}\rightarrow Tot (Gr^2_n{\cal F}C_*)$, 
one arising from the nulhomotopy of the composite $\Sigma^{n+1}C_{n+3}\rightarrow
\Sigma^{n+1}C_{n+2}\rightarrow \Sigma^{n+1}C_{n+1}$, and one arising from the nulhomotopy of the composite
$\Sigma^{n+1}C_{n+2}\rightarrow \Sigma^{n+1}C_{n+1}\rightarrow Tot (Gr^2_n{\cal F}C_*)$. Hence the homotopy-commutativity of diagram~\ref{diagram 1}
is equivalent to the compatibility-up-to-homotopy of $f_*$ with these nulhomotopies, i.e., that the two nulhomotopies of 
$\Sigma^{n+1}C_{n+2}\rightarrow \Sigma^{n+1}D_{n+1}$ given by the two composites in the diagram
\[ \xymatrix{ \Sigma^{n+1}C_{n+3} \ar[r]^{f_{n+3}}\ar[d] & \Sigma^{n+1}D_{n+3} \ar[d] \\
              \Sigma^{n+1}C_{n+2}                 \ar[d] & \Sigma^{n+1}D_{n+2} \ar[d] \\
              \Sigma^{n+1}C_{n+1} \ar[r]^{f_{n+1}}       & \Sigma^{n+1}D_{n+1} }\]
give rise to a nulhomotopic map $\Sigma^{n+2}C_{n+3}\rightarrow \Sigma^{n+1}D_{n+1}$. (That the
two nulhomotopies of $\Sigma^{n+1}C_{n+2}\rightarrow Tot (Gr^2_n{\cal F}D_*)$ automatically give rise to the 
zero map $\Sigma^{n+2}C_{n+2}\rightarrow Tot (Gr^2_n{\cal F}D_*)$ is actually a restatement of Prop.~\ref{map on gr^2}!)
But the map $\Sigma^{n+2}C_{n+3}\rightarrow \Sigma^{n+1}D_{n+1}$ in question is precisely $\Sigma^{n+1}T(2,n;f_{\bullet})$.
Hence the vanishing of the Toda bracket $T(2,n;f_{\bullet})$ is equivalent to being able to extend the map on
$Gr^2_n$ to a map on $Gr^3_n$.
\end{proof}

The general case is described by the following theorem.

\begin{theorem} \label{main thm} Given a map of simplicial homotopy types $f_{\bullet}:C_{\bullet}\to D_{\bullet}$ as above, the map $f_m:Gr^1_m{\cal F}C_{*}\overset{\simeq}\longrightarrow Gr^1_m{\cal F}D_{*}$ extends to a filtration-preserving map
\[
Tot\left(Gr^k_n{\cal F}C_{*}\right)\overset{}\longrightarrow Tot\left(Gr^k_n{\cal F}D_{*}\right)
\]
for some fixed $k \geq 3$ iff the Toda brackets $T(N,i; f_{\bullet})$ vanish 
for all pairs of integers $(N,i)$ with $2\leq N < k$ and $n\leq i$ and $N+i < n+k$.
Given this vanishing, the next higher-order set of Toda brackets are defined:
\[ T(k, n; f_{\bullet}): \Sigma^{k-1} C_{n+k} \rightarrow D_n \]
is the map given by the two nulhomotopic maps in the square
\[ \xymatrix{   \Sigma^{k-2}C_{n+k} \ar[rr]^{T(k-1,n+1;f_{\bullet})} \ar[d]_{\Sigma^{k-2}d^{C}_{n+k}} & & D_{n+1}  \ar[d]^{d^{D}_{n+1}} \\ 
                \Sigma^{k-2}C_{n+k-1} \ar[rr]^{T(k-1,n; f_{\bullet})} & & D_n. } \]

If the Toda brackets vanish for all orders and degrees, then there exists a filtration-preserving map
\[
Tot\left(C_{*}\right)\overset{}\longrightarrow Tot\left(D_{*}\right)
\]
which on $Gr^1_{\bullet}{\cal F}C_{*}$ agrees with $f_{\bullet}$.

Finally, if $f_{n}$ is a weak equivalence for each $n$, then the map on $Tot \left(Gr^{k}_m\cal{F}C_*\right)$ is also a weak equivalence whenever it is defined (i.e.,
whenever all the appropriate Toda brackets, described above, vanish).
\end{theorem}
\begin{proof}
Essentially the same as in Prop.~\ref{gr^3 map}. We work by induction: assume the theorem is true for all $\ell\leq k$ for some $k$.
In other words, assume that when $\ell\leq k$, the homotopy map $f_{\bullet}$ induces a map $Tot (Gr^{\ell}_{m}\mathcal{F}C_*) {\rightarrow} 
Tot (Gr^{\ell}_{m}\mathcal{F}D_*)$
if and only if all the Toda brackets $T(N,n;f_{\bullet})$ vanish for $2\leq N <\ell$ and $n\leq i$ and $N+i< n+\ell$.
We want to show that the same statement then holds for $k+1$.
Under our assumptions, we have a partially-defined map of geometric realization towers:
\begin{equation}\label{general map of geom realization towers} \xymatrix{
 \Sigma^{n+1}C_{n+2} \ar[r]\ar@/^1pc/[rr] & \Sigma^{n} C_{n+1} \ar@/_1pc/[rr] \ar[d] & \Sigma^{n+1}D_{n+2} \ar[r] & \Sigma^{n} D_{n+1}  \ar[d] \\
 \Sigma^{n+2}C_{n+3} \ar[r]\ar@/^1pc/[rr] & Tot ( Gr^2_{n}{\cal F}C_*) \ar[d]\ar@/_1pc/[rr] &  \Sigma^{n+2}D_{n+3} \ar[r] & Tot ( Gr^2_{n}{\cal F}D_*) \ar[d] \\
   & \vdots\ar[d]  &  & \vdots\ar[d] \\
 \Sigma^{n+k-1}C_{n+k} \ar[r]\ar@/^1pc/[rr] & Tot ( Gr^{k-1}_{n}{\cal F}C_*) \ar[d]\ar@/_1pc/[rr] &  \Sigma^{n+k-1}D_{n+k} \ar[r] & Tot ( Gr^{k-1}_{n}{\cal F}D_*) \ar[d] \\
 \Sigma^{n+k}C_{n+k+1} \ar[r] & Tot ( Gr^k_{n}{\cal F}C_*) \ar[d]\ar@/_1pc/[rr] &  \Sigma^{n+k}D_{n+k+1} \ar[r] & Tot ( Gr^k_{n}{\cal F}D_*) \ar[d] \\
                            & Tot ( Gr^{k+1}_{n}{\cal F}C_*) &  & Tot ( Gr^{k+1}_{n}{\cal F}D_*) .}\end{equation}
The problem of extending this map to the bottom-most stage of the towers is exactly the problem of extending the existing map from
$Tot (Gr^k_n{\cal F}C_*)\rightarrow Tot (Gr^k_n{\cal F}D_*)$ to 
$Tot (Gr^{k+1}_n{\cal F}C_*)\rightarrow Tot (Gr^{k+1}_n{\cal F}D_*)$.
The maps $\Sigma^{n+k}C_{n+k+1}\rightarrow Tot ( Gr^k_{n}{\cal F}C_*)$ and $\Sigma^{n+k}D_{n+k+1}\rightarrow Tot ( Gr^k_{n}{\cal F}D_*)$
at the bottoms of these towers arise from the two nulhomotopies of 
$\Sigma^{n+k-1}C_{n+k+1}\rightarrow Tot (Gr^k_n{\cal F}C_*)$ and the two nulhomotopies of 
$\Sigma^{n+k-1}D_{n+k+1}\rightarrow Tot (Gr^k_n{\cal F}D_*)$, and so the existence of the desired map at the 
bottom of the towers in diagram~\ref{general map of geom realization towers} is equivalent to the 
nulhomotopies of $\Sigma^{n+k-1}C_{n+k+1}\rightarrow Tot (Gr^k_n{\cal F}C_*)$ being compatible with the nulhomotopies
of $\Sigma^{n+k-1}D_{n+k+1}\rightarrow Tot (Gr^k_n{\cal F}D_*)$, that is, it is equivalent to the two composite maps
\[ \xymatrix{   \Sigma^{n+k-1}C_{n+k+1} \ar[rr]^{\Sigma^nT(k,n+1;f_{\bullet})} \ar[d]_{d^{C}_{n+k+1}} & & \Sigma^n D_{n+1}  \ar[d]^{d^{D}_{n+1}} \\  
                \Sigma^{n+k-1}C_{n+k} \ar[rr]^{\Sigma^nT(k,n; f_{\bullet})} &  & \Sigma^n D_n } \]
giving rise to a map $\Sigma^{n+k} C_{n+k+1}\rightarrow \Sigma^n D_n$ which is nulhomotopic. But this map
$\Sigma^{n+k} C_{n+k+1}\rightarrow \Sigma^n D_n$ is precisely $\Sigma^n T(k+1,n;f_{\bullet})$.
So the vanishing of the Toda bracket $T(k+1,n;f_{\bullet})$ occurs precisely when the partially-defined map of geometric realization towers in 
diagram~\ref{general map of geom realization towers} extends to a map $Tot ( Gr^{k+1}_{n}{\cal F}C_*)\rightarrow Tot (Gr^{k+1}_n{\cal F}D_*)$.

We note that whenever a map of geometric realization towers as in 
diagram~\ref{general map of geom realization towers} exists, one notices that each induced map $Tot (Gr^{i+1}_n{\cal F}C_*)\rightarrow Tot (Gr^{i+1}_n{\cal F}D_*)$
in the tower is the map induced on cofibers of horizontal maps of a square
\[ \xymatrix{ \Sigma^{n+i} C_{n+i+1} \ar[r]\ar[d] & Tot (Gr^{i}_n{\cal F}C_*) \ar[d] \\
\Sigma^{n+i} D_{n+i+1} \ar[r] & Tot (Gr^{i}_n{\cal F}D_*) ,
}\]
so if one knows that the vertical maps are weak equivalences, then so is the induced map
$Tot (Gr^{i+1}_n{\cal F}C_*)\rightarrow Tot (Gr^{i+1}_n{\cal F}D_*)$. By an obvious induction we get that, if $f_n$ is a weak equivalence for each $n$,
then so is each map $Tot (Gr^{i+1}_n{\cal F}C_*)\rightarrow Tot (Gr^{i+1}_n{\cal F}D_*)$ whenever it is defined.
\end{proof}

Now recall that one has both a homological and a cohomological spectral sequence associated to a tower of homotopy cofiber sequences.
We start with the cohomological spectral sequence. Suppose one chooses an object $S$ of $\mathcal{C}$ and considers the representable functor
$H: (Ho (\mathcal{C}))^{op}\rightarrow Ab$ given by $H(-) = [-, S]$.
This functor sends each triangle
\[ X \rightarrow Y \rightarrow Z \rightarrow \Sigma X\]
in $Ho (\mathcal{C})$ to a long exact sequence of abelian groups
\[ \dots \rightarrow H(\Sigma^{} X) \rightarrow H(Z)\rightarrow H(Y)\rightarrow H(X)\rightarrow H(\Sigma^{-1} Z) \rightarrow \dots \]
and, as a consequence, applying $H$ to the tower of homotopy cofiber sequences (i.e., triangles)
\begin{equation}\label{geom realization tower for ss}
\xymatrix{
\Sigma^{n} C_{n+1} \ar[r] & 0 \ar[d] \\
\Sigma^{n+1} C_{n+2} \ar[r] & \Sigma^{n+1} C_{n+1} \ar[d] \\
\Sigma^{n+2} C_{n+3} \ar[r] & Tot (Gr^2_n FC_{\bullet}) \ar[d] \\
 & \vdots\ar[d]  \\
\Sigma^{n+k-1} C_{n+k} \ar[r] & Tot (Gr^{k-1}_n FC_{\bullet}) \ar[d] \\
                                & Tot (Gr^{k}_n FC_{\bullet}) 
}\end{equation} 
yields an exact couple and hence a spectral sequence.
If we let $E_1^{s,t} = H(\Sigma^tC_s)$ if $n+1\leq s\leq n+k$ and $0$ otherwise, then the associated spectral sequence has differentials
\[ d_r^{s,t}: E_r^{s,t}\rightarrow E_r^{s+r,t+r-1}\]
and $d_1^{s,t}: H(\Sigma^tC_s)\rightarrow H(\Sigma^tC_{s+1})$ coincides with the map $H(\Sigma^t d^{C}_{s+1})$.
The spectral sequence converges strongly to $H(\Sigma^{t-s} Tot (Gr^k_n FC_{\bullet}))$.

If we instead have the infinite tower
\begin{equation}\label{infinite geom realization tower for ss}
\xymatrix{
 \Sigma^{-1}C_{0} \ar[r] & 0 \ar[d] \\
 C_{1} \ar[r] &   C_{0} \ar[d] \\
\Sigma^{} C_{2} \ar[r] & Tot ( F_1C_{\bullet}) \ar[d] \\
\Sigma^{2} C_{3} \ar[r] & Tot ( F_2C_{\bullet}) \ar[d] \\
 & \vdots
 }\end{equation}
then our spectral sequence has $E_1^{s,t}\cong H(\Sigma^t C_s)$ and
converges conditionally to $H(\Sigma^{t-s}Tot (C_{\bullet}))$.

We also have a homological spectral sequence. In order to get it to compute $Tot (Gr^k_n FC_{\bullet})$ and not $Tot (Gr^n_n FC_{\bullet})\simeq 0$, we first
have to ``dualize'' the tower of diagram~\ref{geom realization tower for ss}
by taking the levelwise homotopy cofiber of the map from that tower into the tower of homotopy cofiber sequences
\[ 
\xymatrix{
0 \ar[r] & Tot (Gr^k_n FC_{\bullet}) \ar[d] \\
0 \ar[r] & Tot (Gr^k_n FC_{\bullet}) \ar[d] \\
 & \vdots\ar[d]  \\
0 \ar[r] & Tot (Gr^k_n FC_{\bullet}) \ar[d] \\
 &  Tot (Gr^k_n FC_{\bullet}) .
}\]
We write $T_i$ for the cofiber of the map $Tot (Gr^i_n FC_{\bullet}) \rightarrow Tot (Gr^k_n FC_{\bullet})$. The levelwise cofiber tower we now have is the tower
of homotopy cofiber sequences
\begin{equation}\label{dual geom realization tower for ss}
\xymatrix{
\Sigma^{n+1}C_{n+1} \ar[r] & Tot (Gr^k_n FC_{\bullet} ) \ar[d] \\
\Sigma^{n+2} C_{n+2} \ar[r] & T_0 \ar[d] \\
\Sigma^{n+3} C_{n+3} \ar[r] & T_1 \ar[d] \\
 & \vdots\ar[d]  \\
\Sigma^{n+k} C_{n+k} \ar[r] & T_{k-1} \ar[d] \\
 & T_k \simeq 0.}\end{equation}

If $S$ is an object of $\mathcal{C}$, we have the co-representable functor
$H: Ho (\mathcal{C})\rightarrow Ab$ given by $H(-) = [S, -]$. Again, this functor sends triangles in $Ho (\mathcal{C})$ to long exact sequences,
so applying $H$ to the tower of homotopy cofiber sequences of diagram~\ref{dual geom realization tower for ss} yields an exact couple and hence a spectral sequence.
If we let $E^1_{s,t} = H(\Sigma^tC_s)$ if $n+1\leq s\leq n+k$ and $0$ otherwise, then the associated spectral sequence has differentials
\[ d_r^{s,t}: E^r_{s,t}\rightarrow E^r_{s-r,t-r+1}\]
and $d^1_{s,t}: H(\Sigma^tC_s)\rightarrow H(\Sigma^tC_{s-1})$ coincides with the map $H(\Sigma^t d^{C}_{s})$.
The spectral sequence converges strongly to $H(\Sigma^{t-s} Tot (Gr^k_n FC_{\bullet}))$.
If we instead dualize the tower of diagram~\ref{infinite geom realization tower for ss} and then apply $H$, the resulting spectral sequence
has $E_1^{s,t}\cong H(\Sigma^t C_s)$ and
converges conditionally to $H(\Sigma^{t-s} Tot (C_{\bullet}))$.
\vskip.1in

The reason we are describing these spectral sequences is their relationship to the Toda brackets of Thm.~\ref{main thm}. In the following theorem,
in order to avoid having to write everything twice,
we use the cohomological notation $E_r^{s,t}$ for our spectral sequence even though the theorem also applies equally well to the homological spectral sequence.
\begin{theorem}\label{thm:brackets-diff} Let $f_{\bullet}:C_{\bullet}\to D_{\bullet}$ be a map of totalizable simplicial homotopy types in $Ho (\mathcal{C})$. 
Suppose $H$ is a cohomological functor $H(-) = [-, S]$ {\em or} a homological functor $H(-) = [S, -]$ on $Ho (\mathcal{C})$, as above.
Then we have the spectral sequences
\begin{equation}\label{ss for C} E_1^{s,t} \cong H(\Sigma^t C_s) \Rightarrow H(\Sigma^{t-s} Tot (C_{\bullet}))
\end{equation}
and
\begin{equation}\label{ss for D} E_1^{s,t} \cong H(\Sigma^t D_s) \Rightarrow H(\Sigma^{t-s} Tot (D_{\bullet})).
\end{equation}
The map $f_{\bullet}$ induces a map from the $E_1$-page of spectral sequence~\ref{ss for C} to the $E_1$-page of spectral sequence~\ref{ss for D}. This map
commutes with the differential $d_1$. 

Suppose $k>2$ and the Toda brackets $T(N,i; f_{\bullet})$ vanish 
for all pairs of integers $(N,i)$ with $2\leq N < k$ and $n\leq i$ and $N+i < n+k$. Then the map $f_{\bullet}$ induces a map from the $E_r$-page 
of spectral sequence~\ref{ss for C} to the $E_r$-page of spectral sequence~\ref{ss for D} for every $r< k$. This map
commutes with the differential $d_r$. 
\end{theorem}
\begin{proof}
That $f_{\bullet}$ always induces a map on the $E_1$-pages of the spectral sequences is clear from the fact that the differential $d_1$ in these spectral
sequences is precisely the differential in the alternating sum chain complex object for $C_{\bullet}$ or $D_{\bullet}$.

For the remaining claim in the statement of the theorem, we work by induction.
Suppose $k>2$ and the Toda brackets $T(N,i; f_{\bullet})$ vanish 
for all pairs of integers $(N,i)$ with $2\leq N < k$ and $n\leq i$ and $N+i < n+k$. Furthermore, suppose we know that this implies that $f_{\bullet}$
induces a map, for all $r < k-1$, from the $E_r$-term of the cohomological spectral sequence~\ref{ss for C} to the $E_r$-term of the cohomological spectral 
sequence~\ref{ss for D}. We want to know that we then get an induced map on the $E_k$-terms which commutes with the differentials.
(We work with the cohomological spectral sequences, but the proof for the homological spectral sequences in strictly dual.)
Since the inductive hypothesis implies we have a well-defined map on the $E_{k-1}$-terms commuting with the differentials, on passing to cohomology 
we get a well-defined map of $E_k$-terms, and we must check that it commutes with the differentials. But $E_k^{s,t}$ in spectral sequence~\ref{ss for C}
is a subquotient of $H(\Sigma^t C_s)$ in which every element's image under the boundary map $H(\Sigma^t C_s)\rightarrow H( \Sigma^{t-s} Tot  F_sC_{\bullet})$ 
lies in the image of the map $H(\Sigma^{t-s} Tot  F_{s+k-1}C_{\bullet})\rightarrow H(\Sigma^{t-s} Tot  F_sC_{\bullet})$, and the differential
$d_k^{s,t}$ is just the composite
\[ E_k^{s,t}\rightarrow H(\Sigma^{t-s} Tot  F_{s+k-1}C_{\bullet})\rightarrow E_k^{s+k, t+k-1} \]
where the right-hand map is induced by the map $\Sigma^{t+k-1} C_{s+k}\rightarrow \Sigma^{t-s}Tot  F_{s+k-1} C_{\bullet}$.
The vanishing of $T(k,s; f_{\bullet})$ is precisely what we need in order to know that the square
\[ \xymatrix{ \Sigma^{t+k-1} C_{s+k}\ar[r]\ar[d]_{\Sigma^{t+k-1}f_{\bullet}} & \Sigma^{t-s}Tot  F_{s+k-1} C_{\bullet} \ar[d] \\
\Sigma^{t+k-1} D_{s+k}\ar[r]  &\Sigma^{t-s}Tot  F_{s+k-1} D_{\bullet} }\]
homotopy-commutes and hence that the induced map on $E_k$-pages commutes with the $d_k$ differentials.
\end{proof}

\section{Non-triviality of the obstruction}
\vskip.2in

For an algebra $A$, we write $CH_*(A)$ resp.\ $CC_*(A)$ for the Hochschild resp.\ cyclic complex of $A$, with conventions for the differentials and cyclic structure following that in [L]; $HH_*(A)$ resp.\ $HC_*(A)$ are their respective homology groups.
When $A$ is the complex group algebra $\mathbb C[\pi]$, there are well-known decompositions of $CH_*(\mathbb C[\pi])$ and $CC_*(\mathbb C[\pi])$ as direct sums of subcomplexes, indexed on $<\pi> =$ the set of conjugacy classes of $\pi$, which induce corresponding decompositions in homology:
\begin{gather*}
CH_*(\mathbb C[\pi])\cong \bigoplus_{<x>\in <\pi>} CH_*(\mathbb C[\pi])_{<x>}\\
CC_*(\mathbb C[\pi])\cong \bigoplus_{<x>\in <\pi>} CC_*(\mathbb C[\pi])_{<x>}\\
HH_*(\mathbb C[\pi])\cong \bigoplus_{<x>\in <\pi>} HH_*(\mathbb C[\pi])_{<x>}\\
HC_*(\mathbb C[\pi])\cong \bigoplus_{<x>\in <\pi>} HC_*(\mathbb C[\pi])_{<x>}
\end{gather*}

Moreover, for each conjugacy class associated to an element of infinite order (or {\it non-elliptic class}), there are isomorphisms
\begin{eqnarray}
CH_*(\mathbb C[G])_{<x>}\cong C_*(BC_x;\mathbb C)\label{eqn:conj1}\\
CC_*(\mathbb C[G])_{<x>}\cong  C_*(B(C_x/(x));\mathbb C)\label{eqn:conj2}
\end{eqnarray}
where $x$ is element representing the conjugacy class $<x>$, $C_x$ denotes the centralizer of $x\in\pi$ and $(x)\subset C_x$ the infinite cyclic subgroup of $C_x$ generated by $x$. This identification, due to Burghelea [B], has been fundamental in understanding the structure of the Hochschild and cyclic homology groups of the group algebra. The isomorphism in (\ref{eqn:conj2}) arises from the isomorphism in (\ref{eqn:conj1}), which holds for all conjugacy classes. However, this identification  involves a choice of element $x$ among the set of elements conjugate to $x$; as we shall see, it is impossible to make this choice in a way compatible, up to higher coherence homotopies, with respect to a collection of homomorphisms between two groups.
\vskip.2in

We recall how these equivalences are constructed. Given $x\in G$, write $S_{<x>}$ for the subset of elements in $G$ conjugate to $x$. There is a natural action of $G$ on $S_{<x>}$ given by $g\circ y := g^{-1}yg$. We write $N^{cy}(G)$ for the cyclic bar construction on $G$; this is the cyclic simplicial set with
\begin{gather*}
N^{cy}(G)_n = G^{n+1};\\
\partial_i(g_0,\dots,g_n) = (g_0,\dots,g_i g_{i+1},\dots,g_n),\quad 0\le i\le n-1,\\
\partial_n(g_0,\dots,g_n) = (g_ng_0,g_1,\dots,g_{n-1}),\\
s_j(g_0,\dots,g_n) = (g_0,\dots,g_j,1,g_{j+1},\dots,g_n)
\end{gather*}
and cyclic structure given by
\[
t_{n}(g_0,\dots,g_n) = (g_n,g_0,\dots,g_{n-1})
\]
There is a functorial equivalence
\[
C_*(N^{cy}(G);\mathbb C)\cong CH_*(\mathbb C[G])
\]
Moreover, the decomposition of $CH_*(\mathbb C[G])$ into summands indexed on conjugacy classes arises from the decomposition of $N^{cy}(G)$ into a disjoint sum of path components
\[
N^{cy}(G)\cong \underset{<x>\in <G>}{\coprod} N^{cy}(G)_{<x>}
\]
where for each $<x>$, $N^{cy}(G)_{<x>}$ is the simplicial subset of $N^{cy}(G)$ given by
\[
(N^{cy}(G)_{<x>})_n := \{(g_0,\dots,g_n)\ |\ g_0g_1\cdot\dots\cdot g_n\in S_{<x>}\}
\]
Denoting the non-homogeneous bar resolution of $G$ by $EG$, there is an isomorphism of simplicial sets (compare [L, Prop. 7.4.2])
\begin{eqnarray}
S_{<x>}\underset{G}{\times} EG\overset{\cong}{\longleftrightarrow} N^{cy}(G)_{x},\label{eqn:isom1}\\
(g_1g_2\cdot\dots\cdot g_ng_0;[g_1,g_2,\dots,g_n])\leftrightarrow (g_0,g_1,\dots,g_n)\nonumber
\end{eqnarray}
which sum together over conjugacy classes to induce a simplicial isomorphism
\begin{equation}\label{eqn:isom1global}
S(G)\underset{G}{\times} EG\overset{\cong}{\longleftrightarrow} N^{cy}(G)
\end{equation}
where $S(G) = G$, but with $G$-action given by $g\circ s = g^{-1}sg,s\in S(G)$. Next, for any given element $y\in S_{<x>}$, there is an equivariant isomorphism of $G$-sets
\begin{equation}\label{eqn:isom2}
p_y: C_{y}\backslash G\overset{\cong}{\longrightarrow} S_{<x>},\quad (C_y)g\mapsto g^{-1}yg
\end{equation}
This in turn induces an isomorphism of simplicial sets
\begin{equation}\label{eqn:isom3}
(C_{y}\backslash G)\underset{G}{\times} EG\overset{\cong}{\longleftrightarrow} S_{<x>}\underset{G}{\times} EG
\end{equation}
Finally, the inclusion $C_{y}\hookrightarrow G$ induces a weak equivalence
\begin{equation}\label{eqn;isom4}
BC_{y} = (C_{y}\backslash C_{y})\underset{C_{y}}{\times} EC_{y}\overset{\simeq}{\hookrightarrow} (C_{y}\backslash G)\underset{G}{\times} EG
\end{equation}
The composition
\[
BC_y\to N^{cy}(G)_{<x>}
\]
is therefore a weak equivalence, as well as a map of cyclic simplicial sets, where the cyclic structure on the left is given by the \lq\lq twisted nerve\rq\rq construction detailed in [L,\S 7.3.3] (in the notation of that source, we would write $B(C_y,y)$ instead of just $BC_y$). From this cyclic simplicial weak equivalence, one derives the usual identification of the non-elliptic summands in $CC_*(\mathbb C[G])$ as in (\ref{eqn:conj2}).
\vskip.2in

With respect to naturality, a problem with this construction occurs in (\ref{eqn:isom2}) and (\ref{eqn:isom3}) where the choice of $y$ is made, since this choice cannot be done in a functorial way unless $<x> = <id>$. 
Fixing a choice of $y\in S_{<x>}$ amounts to choosing a basepoint for the non-basepointed discrete space $S_{<x>}$. In what follows, a {\it free simplicial group} refers to a simplicial group which is degreewise free.

\begin{lemma}\label{lemma:one} Suppose $(\Gamma\hskip-.03in_{\bullet})$ is a free simplicial group. Then there is a natural map of graded simplicial sets
\begin{equation}
\left\{[n]\mapsto \underset{<x>\in <\Gamma_n>}{\coprod} BC_{<x>}\right\}_{n\ge 0}\overset{F(\Gamma\hskip-.03in_{\bullet})}{\longrightarrow}
\left\{[n]\mapsto \underset{<x>\in <\Gamma_n>}{\coprod} N^{cy}(\Gamma_n)_{<x>}\right\}_{n\ge 0}
\end{equation}
which, for each $n$ and $<x>\in<\Gamma_n>$, restricts to a weak equivalence of cyclic simplicial sets
\[
BC_{<x>}\overset{\simeq}{\hookrightarrow} N^{cy}(\Gamma_n)_{<x>}
\]
where $C_{<x>}$ is a canonical model for the centralizer subgroup $C_{x}$. Both the domain and range are simplicial spaces (i.e., bisimplicial sets), and for each iterated simplicial map $\lambda:\Gamma_n\to \Gamma_m$, there is a diagram
\begin{equation}\label{diag:diag1}
\xymatrix{
\ar@{}[dr]| *+[F-:<200pt>]{\simeq}
 \underset{<x>\in <\Gamma_n>}{\coprod} BC_{<x>}\ar[d]^{\lambda_*}\ar@{^{(}->}[r]^(0.6){\simeq} &
N^{cy}(\Gamma_n)\ar[d]^{\lambda_*}\\
 \underset{<x>\in <\Gamma_m>}{\coprod} BC_{<y>}\ar@{^{(}->}[r]^(0.6){\simeq} &
N^{cy}(\Gamma_{m})
}
\end{equation}
which commutes up to canonical homotopy
\end{lemma}

\begin{proof}  For each $n\ge 0$ and conjugacy class $<x>\in <\Gamma_n>$, fix a choice of basepoint $x$ for $S_{<x>}$. We consider first the problem of constructing a canonical representative $C_{<y>}$ for the centralizer of an element $y$ when $<y>\ne <id>$. If $<y> = <y'>$, then choosing $h$ such that $y' = y^h$ leads to an isomorphism $C_y\overset{\phi_h}{\underset{\cong}{\longrightarrow}} C_{y'}$ given by $\phi_h(x) = x^h$. However, this isomorphic identification between the two centralizers is determined only up to precomposition with an inner automorphism of $C_y$, as the choice of $h$ is only determined up to right multiplication by an element of $C_y$. Thus a necessary and sufficient condition for $C_{<y>}$ to exist is that the group of inner automorphisms of $C_y$ is trivial; that is, $C_y$ should be abelian when $<y>\ne <id>$. When $<y>\in <F>$, $F$ a free group, $C_y$ is infinite cyclic for $y\ne id$, so the condition holds. In fact, for $y\ne id$, the centralizer $C_y$ is the infinite cyclic subgroup of $F$ generated by $x_y$, where $x_y\in F$ is uniquely defined as the element of shortest length for which $y$ can be written as a power of $x_y$, and the canonical isomorphism between $C_y$ and $C_{y'}$ is the one that sends $x_y$ to $x_{y'}$. Identifying $C_y$ with $C_{y'}$ via this canonical isomorphism for conjugate elements $y$ and $y'$ produces our canonical centralizer group $C_{<y>}$. We can extend this description to all conjugacy classes in $<F>$ by setting $C_{<1>} = C_1 = F$. 
\vskip.1in

There are maps of graded simplicial sets
\begin{equation}\label{eqn:nearend}
\left\{[n]\mapsto \hskip-.2in\underset{<x>\in <\Gamma_n>}{\coprod} \hskip-.1in BC_{<x>}\right\}_{n\ge 0}\hskip-.1in\overset{\simeq}{\longrightarrow}
\left\{[n]\mapsto \hskip-.2in\underset{<x>\in <\Gamma_n>}{\coprod}\hskip-.1in C_{x}\backslash\Gamma_n\underset{\Gamma_n}{\times} E\Gamma_n\right\}_{n\ge 0}
\hskip-.1in\underset{p_*}{\overset{\cong}{\longrightarrow}}
\left\{[n]\mapsto \hskip-.2in\underset{<x>\in <\Gamma_n>}{\coprod}\hskip-.1in N^{cy}(\Gamma_n)_{<x>}\right\}_{n\ge 0}
\end{equation}
\vskip.1in

Where the first map in (\ref{eqn:nearend}) is induced on each summand by the canonical isomorphism $C_{<x>}\cong C_{x}$ described above. For each $\partial_i:\Gamma_n\to\Gamma_{n-1}$ and $<x>\in<\Gamma_n>$ there is a commuting diagram
\begin{equation}\label{diag:diag2}
\xymatrix{
BC_{<x>}\ar[d]^{(\partial_i)_*}\ar@{^{(}->}[r]^(0.4){\simeq} & C_{x}\backslash\Gamma_n\underset{\Gamma_n}{\times} E\Gamma_n\ar[r]^{\cong}_{p_{x}}\ar[d]^{(\partial_i^{h_{x,i}})_*}&
N^{cy}(\Gamma_n)_{<x>}\ar[d]^{(\partial_i^{h_{x,i}})_*}\\
BC_{<\partial_i(x)>}\ar\ar@{^{(}->}[r]^(0.375){\simeq} & C_{\ov{x}_i}\backslash\Gamma_{n-1}\underset{\Gamma_{n-1}}{\times} E\Gamma_{n-1}\ar[r]^{\cong}_{p_{\ov{x}_i}} &
N^{cy}(\Gamma_{n-1})_{<\partial_i(x)>}
}
\end{equation}
where $\ov{x}_i$ denotes the basepoint of $S_{<\partial_i(x)>}\subset\Gamma_{n-1}$, and $h_{x,i}$ satisfies the equation $(\partial_i(x))^{h_{x,i}} = \ov{x}_i$. A similar diagram exists for degeneracy maps. Because of the conjugation by the $\{h_{x,i}\}$ the middle and right-most terms in the sequence of (\ref{eqn:nearend}), equipped with conjugated face and degeneracy maps, may not be bisimplicial sets, but simply graded simplicial sets. However, from the preceding construction used in the definition of $C_{<x>}$, we see that the simplicial identities for compositions of face maps will be satisfied when restricted to the image of the inclusion of $BC_{<x>}$. The same argument applies for the other identities between compositions of face and degeneracy maps, and moreover the simplicial structure is independent of the particular choice of $\{h_{x,i}\}$. The result is a map of graded simplicial sets as indicated, which in each degree is a map of cyclic simplicial sets by [L, Prop. 7.4.5]. Iterating this construction for both face and degeneracy maps, one concludes that for any morphism $\alpha\in Hom_{\Delta}([m],[n])$ and conjugacy class $<x>$, there is a strictly commuting diagram
\begin{equation}\label{diag:diag3}
\xymatrix{
BC_{<x>}\ar[d]^{\Gamma(\alpha)_*}\ar@{^{(}->}[r]^(0.4){\simeq} & C_{x}\backslash\Gamma_n\underset{\Gamma_n}{\times} E\Gamma_n\ar[r]^{\cong}_{p_{x}}\ar[d]^{(\Gamma(\alpha)^{h_{x,\alpha}})_*}&
N^{cy}(\Gamma_n)_{<x>}\ar[d]^{(\Gamma(\alpha)^{h_{x,\alpha}})_*}\\
BC_{<y>}\ar@{^{(}->}[r]^(0.375){\simeq} & C_{\ov{y}}\backslash\Gamma_{m}\underset{\Gamma_{m}}{\times} E\Gamma_{m}\ar[r]^{\cong}_{p_{\ov{y}}} &
N^{cy}(\Gamma_{m})_{<y>}
}
\end{equation}
\vskip.1in
where $\Gamma(\alpha):\Gamma_n\to\Gamma_m$ is the homomorphism corresponding to $\alpha$, $y = \Gamma(\alpha)(x)$, $\ov{y}$ is the basepoint of $S_{<y>}$, and $h_{x,\alpha}\in\Gamma_m$ satisfies the equation $\left(\Gamma(\alpha)(x)\right)^{h_{x,\alpha}} = \ov{y}$. As conjugation by any element of
$\Gamma_m$ induces a self map of $N^{cy}(\Gamma_m)$ canonically homotopic to the identity, we conclude the existence of a canonically homotopy commuting diagram
\begin{equation}\label{diag:diag4}
\xymatrix{
\ar@{}[dr]| *+[F-:<200pt>]{\simeq}
 \underset{<x>\in <\Gamma_n>}{\coprod} BC_{<x>}\ar[d]^{\Gamma(\alpha)_*}\ar@{^{(}->}[r]^(0.6){\simeq} &
N^{cy}(\Gamma_n)\ar[d]^{\Gamma(\alpha)_*}\\
 \underset{<x>\in <\Gamma_m>}{\coprod} BC_{<y>}\ar@{^{(}->}[r]^(0.6){\simeq} &
N^{cy}(\Gamma_{m})
}
\end{equation}
\end{proof}
\vskip.4in
 For a free group $F$, let
\[
\wt{CC}_*(\mathbb C[F]) := C_*(BF;\mathbb C)\otimes CC_*(\mathbb C) oplus
\underset{<id>\ne <x>\in <F>}{\bigoplus} C_*(B(C_{<x>}/(<x>));\mathbb C)
\]
where $C_{<x>}/(<x>)$ denotes the canonical model for the centralizer of $x$ divided by the subgroup $(x)$. This chain complex is simply the cyclic chain complex (in char. $0$) associated to the cyclic simpicial set $\underset{<x>\in <F>}{\coprod} BC_{<x>}$. As we have seen in the proof of the previous Lemma, the association $F\mapsto \wt{CC}_*(\mathbb C[F])$ defines a functor $(fr.gps)\to {\cal C}$ from the category of free groups to $\cal C$.
\vskip.2in

Assume $m\ge 2$, $m$ even. Form a free simplicial group $\Gamma(m)_{\bullet}$ by setting $\Gamma(m)_j = \{id\}$ for $j < m-1$, $\Gamma(m)_{m-1} = \mathbb Z$ on generator $\iota_{m-1}$, and $\Gamma_{m+k}$ the free group on generators $s_{\alpha}(\iota_{m-1})$, where $s_{\alpha}$ ranges over iterated degeneracies from dim.\ $m-1$ to dim.\ $m+k$ when $k\ge 0$. This is a simplcial group model for $\Omega S^m$. Let $A(m)_{\bullet}$ be the abelianization of $\Gamma(m)_{\bullet}$, so that $|A(m)_{\bullet}|\simeq K(\mathbb Z,m-1)$. As $m-1$ is odd, the simplicial group homomorphism $\Gamma(m)_{\bullet}\surj A(m)_{\bullet}$ induced by abelianization is a rational homotopy equivalence, and the map of simplicial complex group algebras $\mathbb C[\Gamma(m)_{\bullet}]\to \mathbb C[A(m)_{\bullet}]$ a weak equivalence.
\vskip.2in

Define $C(m)_{*,\bullet}$, $D(m)_{*,\bullet}$ by
\begin{gather*}
C(m)_{*,\bullet} := \wt{CC}_*(\mathbb C[\Gamma(m)_{\bullet}])\\
D(m)_{*,\bullet} := CC_*(\mathbb C[\Gamma(m)_{\bullet}])
\end{gather*}
By the previous Lemma, both $C(m)_{*,\bullet}$ and $D(m)_{*,\bullet}$ are simplicial objects in $\cal C$, for which there is a homomorphism of graded complexes $\phi_{*,\bullet}: C(m)_{*,\bullet}\to D(m)_{*,\bullet}$ which is a quasi-isomorphism in each degree, and which commutes with face and degeneracy maps up to canonical chain homotopy. Degreewise inclusion of the summand indexed by $<id>$ induces  evident \lq\lq assembly maps\rq\rq
\begin{gather*}
H_*(K(\mathbb Z,m);\mathbb C) = H_*(B\Gamma(m)_{\bullet};\mathbb C)\to H_*(C(m)_{*,\bullet}),\\
H_*(K(\mathbb Z,m);\mathbb C) = H_*(B\Gamma(m)_{\bullet};\mathbb C)\to H_*(D(m)_{*,\bullet})
\end{gather*}

\begin{lemma} For all $m\ge 2$, $H_m(C(m)_{*,\bullet})\ne H_m(D(m)_{*,\bullet})$.
\end{lemma}

\begin{proof} Since the corresponding bicomplexes are positively graded in both coordinates, filtering by rows yields a strongly convergent spectral sequence
\[
\left\{ E^2_{p,q} := H_p(F_{*,q}) \Rightarrow H_{p+q}(F_{*,\bullet})\right\}
\]
for $F = C,D$. In fact, the spectral sequences for both $C_{*,\bullet}$ and $D_{*,\bullet}$ have the same $E^2_{*,*}$-term. In both cases, the image of the canonical generator $\iota_m\in H_m(K(\mathbb Z,m);\mathbb C)$ under the assembly map is represented at the $E^1$-level by the canonical generator $\iota_{1,m-1}\in E^1_{1,m-1}\cong HC_1(\mathbb C[\Gamma(m)_{m-1}]) = HC_1(\mathbb C[\mathbb Z])\cong \mathbb C$. Moreover, in both cases, this element survives to a non-zero element in $E^2_{1,n-1}$. Now the bicomplex $C(m)_{*,*}$ satisfies the property that for each $n$, $0 = d^0_{1,n}: C(m)_{1,n}\to C(m)_{0,n}$. In other words, as a bicomplex it can be written as a direct sum $C_{*,*} = C_{0,*}\oplus C_{*,*}/C_{0,*}$. In the corresponding homology spectral sequence, this forces all differentials originating on the $q=0$ line to be zero. In particular, for the spectral sequence converging to $H_*(C(m)_{*,\bullet})$, one has $0 = d^2_{0,m+1}:E^2_{0,m+1}\to E^2_{1,m-1}$, implying $\iota_{1,n-1}$ survives to a non-zero element in $E^3_{1,m-1} = E^{\infty}_{1,m-1}$, so that $H_m(C(m)_{*,\bullet})\cong\mathbb C$. On the other hand, in the spectral sequence converging to $H_*(D(m)_{*,\bullet})$, the element $\iota_{1,n-1}$ must be hit by the differential $d^2_{0,m+1}$. In fact, $H_*(D_{*,\bullet})\cong HC_*(\mathbb C[\Omega K(\mathbb Z,m)])$ by the above discussion (where $\Omega K(\mathbb Z,m)$ denotes any simplicial group rationally homotopy equivalent to $K(\mathbb Z,m)$). But for $m\ge 2$, $HC_m(\mathbb C[\Omega K(\mathbb Z,m)]) = 0$, with the canonical generator $0\ne \iota_m\in HH_m(\mathbb C[\Omega K(\mathbb Z,m)])$ lying in the image of the $B:HC_{m-1}(\mathbb C[\Omega K(\mathbb Z,m)])\to HH_m(\mathbb C[\Omega K(\mathbb Z,m)])$ in the Connes-Gysin sequence. 
\end{proof}

\begin{corollary} For each $m\ge 2$, the $(m+1)^{st}$ first-order Toda bracket associated to the homotopy chain map $C(m)_{*,\bullet}\to D(m)_{*,\bullet}$ is non-zero.
\end{corollary}

\begin{proof} This is an immediate consequence of Theorem \ref{thm:brackets-diff}.
\end{proof}

 Note that the above phenomenon is fundamentally a non-commutative one. In fact, revisiting the proof of Lemma \ref{lemma:one} we see that the canonical models for the centralizer subgroups are also defined when the simplicial group is degreewise abelian. However, in this case, the diagram in (\ref{diag:diag1}) commutes not just up to canonical homotopy, but on the nose. In other words, if $A_{\bullet}$ is a simplicial abelian group, there is a homomorphism of simplicial chain complexes (not just graded complexes)
\[
\{[n]\mapsto \wt{CC}_*(\mathbb C[A_n])\}_{n\ge 0}\to \{[n]\mapsto CC_*(\mathbb C[A_n])\}_{n\ge 0}
\]
(with the left-hand side defined exactly as above) which is a quasi-isomorphism in each degree, hence a quasi-isomorphism of total complexes. From this we can also conclude that the functor
\[
\Gamma\hskip-.03in_{\bullet}\mapsto \wt{CC}_*(\mathbb C[\Gamma\hskip-.03in_{\bullet}])
\]
defined for simplicial groups which are either degreewise free or degreewise abelian, does not admit an extension to a homotopy functor from the category $S_{\bullet}(gp.s)$ of simplicial groups to $S_{\bullet}{\cal C}$, for this last observation implies that the abelianization map $\Gamma(m)_{\bullet}\to A(m)_{\bullet}$, which is a weak equivalence, does not induce a quasi-isomorphism when precomposed with $\wt{CC}_*(_-)$.

\end{document}